\documentclass{amsart}
\usepackage{changebar}

\usepackage{amsthm}

\newtheorem{lemma}{Lemma}
\newtheorem{theorem}{Theorem}
\newtheorem{proposition}{Proposition}

\newtheorem{remark}{Remark}
\newtheorem{fact}{Fact}

\newcommand{\disc}{\mbox{disc}}
\newcommand{\prk}{p\hbox{-rk} }

\newcommand{\zrk}{\Z\hbox{-rk} }
\newcommand{\rk}{\hbox{rk} }

\newcommand{\nth}[1]{$#1 {\rm - th }$}

\newcommand{\rem}{{\rm rem } \ }

\newcommand{\Z}{\mathbb{Z}}

\newcommand{\Tr}{\mbox{\bf Tr}}

\newcommand{\rg}[1]{\mbox{\bf #1}}
\newcommand{\eu}[1]{\mathfrak{#1}}
\newcommand{\id}[1]{\mathcal{#1}}
\newcommand{\Gal}{\mbox{ Gal }}
\newcommand{\Mat}{\mbox{ Mat }}

\newcommand{\rf}[1]{(\ref{#1})}

\newcommand{\Norm}{\mbox{\bf N}}
 
\newcommand{\lchooses}[2]{\left( \frac{#1}{#2 } \right)}

\newcommand{\cog}{\hbox{Cog}}

\newcommand{\F}{\mathbb{F}}

\newcommand{\K}{\mathbb{K}}

\newcommand{\KL}{\mathbb{L}}
\newcommand{\M}{\mathbb{M}}
\newcommand{\Q}{\mathbb{Q}}
\newcommand{\R}{\mathbb{R}}
\newcommand{\C}{\mathbb{C}}
\newcommand{\N}{\mathbb{N}}

\def\ra{\rightarrow}
\newcommand{\ran}{\rangle}
\newcommand{\lan}{\langle}
\newcommand{\wh}{\widehat}
\newcommand{\td}{\widetilde}
\newcommand{\subsetneq}{\subset}
\newcommand{\ol}{\overline}

\begin{document}
{\obeylines \small
\vspace*{0.2cm}
\smallskip}\title[Lower bounds for FLT2] {Double exponential lower bounds for possible solutions in the Second Case of the Fermat Last Theorem} 
\author{Preda Mih\u{a}ilescu and Michael T. Rassias} 
\address[P. Mih\u{a}ilescu]{Mathematisches Institut der Universit\"at 
G\"ottingen}
\email[P. Mih\u{a}ilescu]{preda@uni-math.gwdg.de} 
\address[M. Th. Rassias]{Institute of Mathematics, University of Zurich, CH-8057, Zurich, Switzerland}
\email[M. Th. Rassias]{michail.rassias@math.uzh.ch}
\date{Version 1.0 \today}
\vspace{2.0cm}
\begin{abstract}
In a recent paper, the first author provided some lower bounds to solutions of the equations of Fermat and Catalan, based on local power series 
developments at the ramified prime of a prime cyclotomic extension. Although both equations have in fact been proved not to have any unknown
solutions, these improved bounds are interesting in the context of a new effective abc inequality announced in the paper \cite{MFHMP} 
based Mochizuki's  \cite{Mo}[IUT-IV, Theorem A]. In this paper we provide a strengthening of the lower bound for FLT2, which is necessary in order
to take advantage of the best upper bounds for primes $p$ for which it was verified on a computer that FLT2 has no solutions. 

\end{abstract} 
\maketitle

\section{Introduction and notations}
This paper improves upon the lower bound proved in the recent paper \cite{Mi3}, by extending upon the method used there.
Since the initial steps of the argument are similar, we use parts of the introductory facts from \cite{Mi3}, in order to 
introduce the basic notions on the basis of which we can then explain our strategy and complete the proofs.
The improvements are quite impressive , compared to previous results, and they were made possible by a very useful
new insight that strengthens the approach taken in \cite{Mi3}. This will be shortly described at the end of the chapter, in 
the {\em plan of the paper}. At that point, sufficient notations are introduced, allowing an easier presentation of the new ideas.
The paper being short, the interested reader can also jump to the paper presentation, catching up during the reading with 
notation that may be unfamiliar. 

Throughout this paper, $p$ is an odd prime. We state here, for ease of reference, the  equation of interest:
\begin{eqnarray}
\label{flt}
x^p + y^p & = & z^p; \quad x, y, z \in \Z_{\neq 0}, \quad ( x,y,z ) = 1, \quad \hbox{and $p$ is an odd prime}.
\end{eqnarray}

The following known relations \cite{Ri1}( \S IV.1, {\bf 1B}, p. 54 ), which hold iff there is a solution for FLT2, and 
 in which one may assume the choice $x, z > 0$ and $| x | > | y |$,  will be of interest:
\begin{eqnarray} 
\label{flt2}
x^p + y^p & = & z^p, \quad p^2 | z, \quad p^{2p-1} | ( x+y ), \\
\label{barl}
\frac{x^p+y^p}{p ( x+y) } & = & \Norm_{\K/\Q} \left(\frac{x+y}{1-\zeta} - y \right) = s^p, 
\end{eqnarray}  
for some $s \in \Z$ and $\zeta \in \C$, a primitive \nth{p} root of unity, 
$\K = \Q[ \zeta ]$.

We shall prove:
\begin{theorem}
\label{main}
Suppose that $( x, y, z )$ is a triple satisfying \rf{flt} and $p | z$, with $p \geq 257$ a prime. Then 
\[ \max( |x|, |y|, |z| ) \geq p^{(5/2)^{p-1}}. \] 
\end{theorem}
Throughout this note, we denote the set of minimal positive representatives of $\F_p$ by $P= \{0,1,\ldots, p-1 \}; P^*= P\setminus \{ 0 \}$,
 and $\zeta$ will be a primitive \nth{p} root of unity; we also write $q= \frac{p-1}{2}$. We let $\K = \Q[ \zeta ]$ be the cyclotomic field, with 
galois group $G = \Gal( \K /\Q )$ and $\sigma_c \in G$ are the automorphisms given by
 $\zeta \mapsto \zeta^c$, for $c \in P^*$; we fix $\sigma = \sigma_g \in G$, an automorphism generating $G$ as a cyclic group. 
 The complex conjugation acting in $\K$ is $\jmath = \sigma_{p-1} = \sigma^{(p-1)/2}$. For $\rg{R} \in \{ \F_p, \Z_p\}$, and
 $\varpi ; G\ra \rg{R}$ denoting the Teichm\"uller character, the {\em orthogonal idempotents} 
$e_k \in \rg{R}[ G ]$ are
\begin{eqnarray}
\label{idem}
e_k = \frac{1}{p-1} \sum_{a=1}^{p-1} \varpi^k(\sigma_a) \cdot \sigma_a^{-1}. 
\end{eqnarray}

 We use the uniformizer $\lambda = 1-\zeta \in \Z[ \zeta ]$,
 that generates the principal prime $\wp \subset \Z[ \zeta ]$ above $p$. It induces $\lambda$-adic expansions of algebraic integers in $\Z[ \zeta ]$,
 so we may write, for instance: 
 \[ \alpha = \sum_{j=0}^{\infty} a_j \lambda^j = a_0 + a_1 \lambda + a_2 \lambda^2 + O( \lambda^3 ) , \quad \forall \alpha \in \Z[\zeta ], \]
where the $a_j \in \{ 0, 1, \ldots, p-1\}$, and only finitely many are not vanishing. And the symbol $O( \lambda^k )$ stands for a remainder, which 
is divisible by $\lambda^k$. The same notation can be used also in $\Z_p[ \zeta ]$.

The Stickelberger element $\vartheta = \frac{1}{p} \sum_{c=1}^{p-1} c \sigma_c^{-1} \in \frac{1}{p} \Z[ G ]$ generates
the Stickelberger ideal  in the group ring of $G$ over the rational integers, by intersecting its principal ideal with $\Z[G ]$, 
according to 
\begin{eqnarray}
\label{stickid}
 I = \vartheta \Z[ G ] \cap \Z[ G ] . 
 \end{eqnarray}
The ideal $I$ has the property of annihilating the class group of $\K$ ( \cite{Wa}, \S 15.1). 
To each ideal $\eu{C} \subset \Z[ \zeta ]$ and each $\theta \in I$, 
the ideal $ \eu{C}^{\theta} \subset A$ is principal, generated by $\gamma \in \Z[\zeta ]$, and 
$\gamma \cdot \overline{\gamma} = \Norm( \eu{C} )^{\varsigma( \theta )}$, for some
integer $\varsigma( \theta ) \in \Z$, which we call the {\em relative weight} of $\theta$. 
There exists a base for $I$ made up by elements of weight one: they are 
called {\em Fueter elements}, e.g. \cite{Mi2} and which are written as
\begin{eqnarray}
\label{fueter} 
 \psi_n &= &\sum_{c \in I_n } n_c \sigma^{-1}_c  = ( \sigma_{n+1}-1-\sigma_n ) \vartheta \in \Z_{\geq 0}[ G ], \\ &&\quad \nonumber 
 n_c + n_{p-c} = 1, \quad n = 1, 2, \ldots, \frac{p-1}{2},
\end{eqnarray}
where the sets $I_n \subset \{1, 2 \ldots, p-1 \}$ verify $I_n \sqcup ( p - I_n ) = \{1, 2 \ldots, p-1 \}$ and are deduced from the definition of $\psi_n$.
Two particular values that we shall use are
\begin{eqnarray}
\label{psi2}
\psi_1 = \sum_{c > p/2} \sigma_{c}^{-1} \quad \hbox{and} \quad \jmath \psi_1= \sum_{c < p/2} \sigma_{c}^{-1}.
\end{eqnarray} 

Thus, for any $\theta \in I$, there is a development
\[ \theta = \sum_{n=1}^{(p+1)/2} \nu_n \psi_n = \sum_{c=1}^{p-1} n_c \sigma_c^{-1}; \quad \nu_n, n_c \in \Z. \]

Numbers $\gamma$ generating the principal ideal $\eu{C}^{\theta}$ are, up to roots of unity, products of Jacobi sums and are called, by 
extension, {\em Jacobi numbers} (\cite{Jh}, \cite{Mi2}). Iwasawa proved in \cite{Iw} that Jacobi numbers verify 
$\gamma \equiv 1 \bmod ( 1 - \zeta )^2$, 
a relation which is used for norming the Jacobi integer generators of the previously mentioned ideals.
See also the introduction of \cite{Mi2} for an exhaustive presentation of properties of the Stickelberger Ideal as used in our context.
\begin{remark}
\label{unicity}
 It is also  proved in \cite{Mi2}, that if $J \subset \Z[\zeta ]$ is some principal ideal that is generated by a Jacobi number $\eu{j} \in \Z[ \zeta ]$ -- 
 so $J = ( \eu{j} )$ --  then this number is uniquely determined by $J$ and it verifies $\eu{j} \cdot \overline{\eu{j}} \in \N$. 
\end{remark}

\subsection{Basic facts in the \nth{p} cyclotomic field}
The left terms of the equation \rf{flt2} factor in the field $\K$ and the assumed solutions 
give raise to specific {\em characteristic numbers and ideals}, which are rich in properties, that we review in the following two facts.

\begin{fact}
\label{aux1}
\begin{itemize}
\item[ A. ] There is a {\em characteristic (algebraic) number}, which encodes the properties of the solutions, in the \nth{p} cyclotomic field
$\K$. This is 
\[ \alpha( x, y )   = \frac{x+y}{1-\zeta} - y.  \]
Since $p^{2p-1} | (x+y)$, this is indeed an integral element. 
\item[ B. ] The galois group $G$ acts on the characteristic number, giving raise to pairwise coprime integral 
elements, for $1 \leq c < d \leq p-1$, we have
\begin{eqnarray*}
( \sigma_c( \alpha ), \sigma_d( \alpha )) = (1) .
\end{eqnarray*}
\item[ C. ] There is a {\em characteristic ideal} $\eu{A} = ( \alpha, s )$ of order dividing $p$. It is related 
to the previously introduced number by the relations:
\begin{eqnarray}
\label{rclassf}
\eu{A}^p & = & ( \alpha ), \quad \Norm( \eu{A} ) = (s).
\end{eqnarray}
\end{itemize} 
\end{fact}
The annihilation of the ideals in \rf{rclassf} by elements $\psi \in I$ of the Stickelberger ideal, leads to some algebraic 
numbers -- in fact, {\em Jacobi numbers}, in the sense explained above --- which depend on $\psi$ and 
which can be developed in local binomial power series, as  a consequence, essentially, of the relations in the above identities. 
Binomial series and the Jacobi numbers depend on $\psi$, and we shall use notations of 
the type $\beta[ \psi ] \in \Z[ \zeta ]$ for the Jacobi numbers, and $f[\psi ] ( T )$ for 
the binomial series associated to annihilation by $\psi$; the use of square brackets
rather than indices, for bringing these dependencies into evidence, is preferable from the typographic point of view; it can be dropped as soon as 
the Stickelberger element associated to a binomial series or Jacobi number is evident in the context. 
We describe these resources for arbitrary $\psi \in I \cap \Z_{\geq 0}[ G ]$. We have:
\begin{fact}
\label{aux2}
\begin{itemize}
\item[ D. ] Suppose that $\psi = \sum_{c=1}^{p-1} n_c \sigma^{-1}_c \in \Z_{\geq 0}[ G ]$ is a positive Stickelberger element. 
The annihilation of the characteristic ideals yields principal ideals ( Jacobi numbers ) as follows:
\begin{eqnarray}
\label{powsf}
( \beta[ \psi ] )   & \eu{A}^{\psi} \subset A; \quad \beta[ \psi ]^p =  \alpha^{\psi}.
\end{eqnarray} 
The Jacobi numbers $\beta[\psi]$ are uniquely determined by these relations, as consequence of Remark \ref{unicity}.
\item[ E. ] Dividing with complex conjugates yields practical expressions for $p$-adic power series developments, as follows:
\begin{eqnarray}
\label{djf}
 \gamma[ \psi ]   =  \frac{\beta[ \psi ]}{\overline{\beta[ \psi ]}}; \quad \gamma[ \psi ]^p = \alpha^{(1-\jmath)\psi} = 
 \left( \frac{1 - \frac{x+y}{y(1-\zeta)}}{1 - \frac{x+y}{y(1-\overline{\zeta})}} \right)^{\psi}.
\end{eqnarray} 
\item[ F. ] Since $\beta[ \psi ] \cdot \bar{\beta}[ \psi ] =  s^{\varsigma( \psi) }$, 
we obtain integral elements in E. by multiplication with constants:
\begin{eqnarray}
\label{2intf}
s^{\varsigma( \psi) } \gamma[ \psi ]  =  \beta^2[ \psi ] \in \Z[\zeta ]
\end{eqnarray}
\end{itemize} 
\end{fact}
The facts gathered above are folklore, or part of the prerequisites proved in several of the papers and books cited above. 
We provide here indications for the proofs. 
\begin{proof}
By \rf{flt2} and \rf{barl}, 
\begin{eqnarray}
\label{pvalf}
v_p( \alpha + y ) = 2p - 1 -\frac{1}{p-1} + l p, \quad \hbox{ for some $l = v_p( z ) - 2 \geq 0$ and $\alpha \in \Z[\zeta]$} .
\end{eqnarray}
The fact that $I( a,b ) = ( \sigma_a( \alpha ), \sigma_b( \alpha )) = (1)$ follows by noting that $I( a, b) = ( x +y,y) = (1) $. Indeed 
\begin{eqnarray*}
y \cdot \frac{\zeta^a - \zeta^b}{1-\zeta^b} &= &\frac{1-\zeta^a}{1-\zeta^b} \sigma_a( \alpha ) -\sigma_b( \alpha )  = 
\varepsilon_1 y \in I( a,b ) , \\
\lambda \cdot ( \sigma_a \alpha - \sigma_b \beta) & = & \frac{\lambda ( \zeta^a - \zeta^b)}{(1-\zeta^a)(1-\zeta^b)} \cdot (x+y ) = \varepsilon_2( x+y) \in I( a,b),
\end{eqnarray*}
where $\varepsilon_{1,2}$ are units, so we also conclude that $y, x+y \in I( a, b)$, as claimed.  The ideal
\[ \eu{A}^p = ( \alpha^p, \alpha^{p-1} s, \ldots, \alpha s^{p-1}, \Norm( \alpha ) ) = ( \alpha ) \cdot \left( \alpha^{p-1}, 
\ldots, \prod_{c \neq 1} \sigma_c( \alpha ) \right) .\]
Since $I( 1, c ) = ( 1 )$, the right-most ideal in the previous identity is the one-ideal, and thus $\eu{A}^p = ( \alpha )$, 
which explains Fact \ref{aux1} for the Fermat equation. 
This completes the proof of the first fact, in the FLT2 case.

In D., the fact that $\eu{A}^{\psi}$ is a principal ideal is a consequence of the fact that the Stickelberg ideal annihilates the class group.
By definition, it is generated by a Jacobi
number, which we denote with $\beta[ \psi ]$. 

By raising to the \nth{p} power, we get from \rf{rclassf} the identity
\[ ( \beta[ \psi ]^p ) = \left( \eu{A}^p \right)^{\psi} = ( \alpha^{\psi} ),    \]
which is an equality of principal ideals generated by Jacobi numbers. It follows from Remark \ref{unicity} that
the identity $ \beta[ \psi ]^p = \alpha^{\psi}$ holds between Jacobi numbers, and this is \rf{powsf}.
The quotient 
\[ \gamma^p[ 1 ] = \alpha/\overline{\alpha} = \frac{1 - \frac{x+y}{y(1-\zeta)}}{1 - \frac{x+y}{y(1-\overline{\zeta})}} \] 
is built such as to cancel $y$, so we obtain a fraction with a nice $p$-adic development, and raising to the power $p$ yields, 
under application of \rf{powsf}, the defining
relation \rf{djf}. This relation is particularly well suited for a $p$-adic development of $\gamma[ \psi ]$, 
in view of the large valuation of $\alpha+y$, noticed in \rf{pvalf}.
In our context it is important to work with algebraic integers, and $\gamma[\psi ]$ is not one. However, 
by multiplying by $s^{\varsigma( \psi ) }$ we do obtain an algebraic integer, whose local power series 
development results herewith from \rf{2intf}. 
\end{proof} 

\subsection{Plan of the paper}
The point F. in Fact \ref{aux2} is the key for bounds found in \cite{Mi3} and also for the present approach. 
The idea was that the map $\gamma : I/( I \cap p \Z[ G ] ) \ra \K^{\times}$ is connected to binomial power extensions
that converge $p$-adically; especially the fact that in the assumption that $p | z $, the valuation $v_p( x+y ) \geq 2p-1$
 allowed already in \cite{Mi3} a substantial improvement upon the lower bounds previously known. The idea is to 
produce linear combinations $\delta = \sum_{\psi \in J} \ell( \psi ) \beta[ \psi ]^2$ of the $p$-adic power series for some 
$\beta^2[\psi] = s^{2 \varsigma( \psi )} \gamma( \psi ) : \psi \in J \subset I$, 
in which the lowest terms vanish: if the linear combination is non vanishing, 
then it is a number divisible by some large power of $p$, say $\delta \equiv 0 \bmod p^N$. The bounds are deduced by comparing 
the absolute value $s = | \beta[ \psi ]^2 |$, the resulting upper bound $| \delta | < L \cdot | J | s $ under the condition $| \delta | \geq p^N$. 
One sees that the quality of the bound depends on the sizes of $L, | J |$ compared to $N$. In \cite{Mi3} we only consider one $G$-orbit
$J = G \psi \subset I^+$, so $| J | = p-1$, and of course, at most $p-1$ coefficients can be brought to vanishing; this is done by 
following an older approach of ours. This  consists in solving full sized regular linear systems, which are homogenous up to one
inhomogenous condition, used for ascertaining that $\delta\neq 0$. The full system - approach thus solves the problem of 
proving $\delta \neq 0$; in exchange, the size of $\log( L ) $ grows quadratically with the coefficient vectors in the system matrix. This limits
the lower bound that we can achieve in this way to an exponent quadratic in $p$, thus $| s | > p^{p^2}$, as shown in \cite{Mi3}.

Considering $A := I^-/( I^- \cap p \Z[ G ])$ -- and identifying the ring, by abuse of notation, with some system of representatives for this quotient, 
we notice that this ring has a large reserve of $G$-orbits, which produce algebraic numbers
with converging $p$-adic power series developments. In addition, in the next chapter we consider the linear independence of
the infinite $p$-adic vectors associated to binomial power series for $\gamma( \psi ); \psi \in A$. Since the obstruction to 
larger lower bounds is the quadratic growth of the bound $\log(L)$ for the solutions of linear systems, a standard approach would be
to consider underdetermined linear systems -- given the fact that $A$ contains numerous $G$-orbits. This is precisely the approach
that we take here; it became only possible due to several new ideas that help dealing with two issues, always arising in similar
contexts of solving linear systems in order to determine coefficients of linear combinations with some pleasant properties, like $\delta$. 
The first problem is that the Siegel box lemma applies to underdetermined systems only in the homogenous case; but we also need
to provide condition ensuring that $\delta \neq 0$. The second obstruction comes from the lack of control over the 
ranks of our linear systems. 

In concrete terms, suppose that we have a collection $J \subset A$ of $G$-orbits and 
\begin{eqnarray}
\label{appro}
 \delta & := & \sum_{\theta \in J} \ell(\theta ) \beta^2[\theta ]; \nonumber \\
  \beta^2[ \theta ] & = & s^{\varsigma( \theta )} \sum_{n \in \N} a_n[ \theta ] T^n; \quad a_n[ \theta ] \in \Z[ \zeta ]; v_p( T ) = 2p-3,
 \end{eqnarray}
 so the power series in the second line above are $p$-adically convergent. Then we wish the $\ell( \theta )$ to fulfill the following 
 expectations:
 \begin{itemize}
 \item[ 1. ] The bound $L = \max_{\theta \in J} | \ell( \theta ) |$ is not too large; more precisely, we wish $\log( L )/\log( p )$ to grow
 at most linearly with $| J |$. 
 \item[ 2. ] We have $\sum_{\theta \in J} \ell( \theta ) a_n[ \theta ] = 0$ for $n < N \sim  \lceil | J |/a \rceil $ for some $a \in \N$.
 \item[ 3. ] Some additional conditions for $n > N$ ensure that $\delta \neq 0$. 
 \end{itemize}
Let $v_k = (a_k[ \theta ] )_{\theta \in J} \in \Q^{ | J | }$ be the vectors of the \nth{k} coefficients of the power series for $\beta$ and
 \[ V_n = \left[ v_k; 0 \leq k \leq n  \right]_{\Q} \subset \Q^{| J |}, \]
be the spaces spanned by the first $n$ such vectors. They have an increasing sequence of 
dimensions $d_n = \dim_{\Q} ( V_n )$, but nothing guarantees for instance strict growth. However, the
 investigation of formal power series and the infinite vectors attached to them give the precise upper bound which is in fact reached
 by the dimensions $d_n$ for large enough $n$. The important breakthrough of this paper consists in ideas allowing to produce the
 inhomogenous conditions in 3. by means of some modified vanishing conditions -- thus allowing still the use of the Siegel box.
 The solution is found by the simple trick of {\em twisting} the vector $v_{N}$ by some small vector $\eta$, thus obtaining an other
 $v' = v_N+\eta$: one can choose $\eta$ such that $\vec{\ell} \perp v'  \Rightarrow \vec{\ell} \not \perp v_N$. The homogenous
 condition $v' \perp \vec{\ell}$ can be used in conjunction with the Siegel box Lemma, and it produces at the same time the inhomogenous
 condition $v_N \not \perp \vec{\ell} $. Along with this core idea, in the practical solution,
  several additional issues need to be taken care of. Since we work
 $p$-adically, a non vanishing term in a power series can be cancelled out by carry -- some additional conditions need to be added, in order
 to avoid this to happen. At the same time, since the dimension $d_k$ may have stationary steps, one must also see for it, that the perpendicularity
 conditions do not become contradictory; these details are dealt with quite naturally and we invite the reader to discover the solutions directly
 in the text. It is also useful to mention that we choose to arrange the coefficients $\ell( \theta )$ in $G$-conjugacy classes, so that 
 $\ell( \sigma \theta ) = \sigma ( \ell( \theta ))$ for $\sigma \in G$ and $\theta \in J$. 
 This explains why the vector spaces $V_n$ are $\Q$-spaces
 and not $\K$-vector spaces. The scalar product becomes concatenation of traces along $G$-orbits of Stickelberger elements. 
 Finally, we choose $J$ close to maximal possible size; in fact, the bound that can be achieved with the present approach will be in the order
 of $p^{p^{(p-1)/4}-a}$ for some small integer $a$. Our exponent is slightly smaller, in order to allow a simple and transparent estimate of 
 the number of independent $G$-cycles in $J$. This difference is irrelevant for the applications mentioned in the introduction, and for which the
 paper is produced: indeed, since the Fermat Conjecture has been proved by Wiles and Taylor now since decades, the interest of such lower 
 bounds depends of the capacity to provide matching upper bounds, and herewith obtain some interesting alternative proofs -- as is done with
 the abc inequality of Fesenko et. al. During the development of this paper, further improvement were found. 
 These lead to a series of separate papers that were completed simultaneously, and which in themselves also provide tight upper bounds, 
 thus extending our methods to effective proofs of more general classes of ternary cyclotomic norm diophantine equations. 
 
 It is interesting to note that specialists in lattices and Minkowski bounds, in generalized Siegel and Bombieri-Vaaler box principles
 and their applications, use similar ideas, for instance in connection with {\em sparse vectors}\footnote{I owe this observation to Lenny Fukshanski,
 who followed closely the development of the lattice related questions and solutions in this paper, and remarked the certain analogy to works
 like \cite{FGK}. This indicates also that the method is both sound and natural.} .
 
\begin{remark}
\label{outlook}
 It is fair to also mention in this presentation of the work, the favorable circumstances specific to FLT2, which herewith produce a limitation
 for the application of this version of the method. One advantageous circumstance consists in the fact that the absolute values $| a_n[ \theta ] |$
 in \rf{appro} grow sensibly slower than $p^{n v_p( T )}$; or, in other words, the valuation $v_p( T )$ is sufficiently large. 
 Once the upper bound $L$ on the $| \ell( \theta) |$ is controlled by the idea described
 above, it is precisely the quotient between these two quantities that accounts for the quality of the lower bounds gained. Finally, $p$-adic 
 development in the second case is special in as much as, one can prove in this case that the binomial series $f_{\theta}$ introduced below, 
 converge precisely to $\gamma[ \theta ]$. In the first case, even if this is true for some choices of $\theta$, the convergence is too slow for
 gaining any bounds. More generally, local power series that converge sufficiently well do exist, but their sum differs from $\gamma[ \theta ]$
 by some erratic roots of unity. It will be shown in subsequent papers how to solve this last obstruction, thus gaining upper bounds for larger families
 of cyclotomic norm equations.
\end{remark} 
\section{Formal power series, function fields and linear spaces of infinite vectors}
We let $\mu = \frac{p^2}{\lambda} \in \Z[\zeta]$ and introduce, for $\theta = \sum_{c=1}^{p-2} n_c \sigma_c^{-1}$,
 some formal power series $f[ \theta ](T) \in \K[[ T ]]$:
\begin{eqnarray}
\label{powser}
f[ k \sigma ]( T ) & := & (1-\mu T)^{k \sigma/n} = 1 + \sum_{n=1}^{\infty} (-\sigma( \mu ))^n \binom{k/p}{n} T^n; \\
f[ \theta ]( T ) & := & \prod_{c=1}^{p-1}  f[ n_c \sigma_c^{-1} ]( T ) := 1 + \sum_{n=1}^{\infty} a_n( \theta ) T^n.  \nonumber.
\end{eqnarray}
The products in the second line are rearranged by increasing powers of $T$, which is possible for formal power series, and also
for uniformly and absolutely convergent evaluations thereof. By definition of the binomial series, we have of course
\begin{eqnarray}
\label{defbin}
 \left( f[ \theta ]( T )\right)^p = ( 1- \mu T )^{\theta}. 
 \end{eqnarray}

One can prove -- see \cite{Mi2} -- that $a_n( \theta ) \in \Z[ \zeta ]$,
and in fact, for a uniform bound $M \geq \sigma_c( \mu )$ for all $c \in P^*$, we have
\begin{eqnarray}
\label{abound}
| a_n( \theta ) | & \leq & M^n \bigg \vert \binom{-w( \theta )/p}{ n } \bigg \vert;
\end{eqnarray}
this bound is derived also in \cite{Mi3}. 

We write $a_n( \theta ) = \sum_{c=1}^{p-1} u_n^{(c)}( \theta ) \zeta^c$ with $u_n^{(c)}( \theta ) \in \Z$, as explained above.
We can in fact divide out the power $e(n) = n - 1 - \left[ \frac{n}{p-1} \right]$ of $p$ out of $a_n$; this still yields an integral element
$ \alpha_n := \frac{a_n}{p^{e(n)}}$.
We define the infinite vectors 
\[ \rg{a}( \theta ) = ( \alpha_n( \theta ) )_{n \in \N} \in \K^{\N}; \quad \rg{u}^{(c)}( \theta ) = 
( u_n^{(c)}( \theta ) )_{n \in \N} \in \Z[ 1/p ]^{\N}, \]
and there is a one-to-one map between $G$-orbits and coefficient vectors:
\begin{eqnarray}
\label{vecmap}
 \{ \rg{a}( \sigma \theta ) \ : \ \sigma \in G \} \quad & \leftrightarrow &\quad \{ \rg{u}^{(c)}( \theta ) \ :\ c = 1,2,\ldots, p-1 \} \\
 p \cdot \rg{u}^{(c)}( \theta ) =  \Tr ( \bar{\zeta}^c \rg{a}( \theta )) - \Tr( \rg{a}( \theta ) ) & \quad & \rg{a}( \sigma \theta  ) = \sigma\left( \sum_{c=1}^{p-1} \rg{u}^{(c)}( \theta ) \zeta^c \right) \nonumber 
 \end{eqnarray}
 
 \begin{fact}
 \label{bounds}
 Let $k = 2 \cdot l$ and $\theta \in I$ have relative weight $\varsigma( \theta ) = k$. Then 
 \begin{eqnarray}
 \label{ubd}
 M_n := \max\left( \vert u_n^{(c)}( \theta ) \vert, \ \vert \alpha_n( \sigma \theta ) \vert \right) < 2\binom{n+l-1}{n} \cdot (p^2/6)^{n+1},
 \end{eqnarray}
 and for $l < p$ we always have $M_n < n^l (2p/3)^{2(n+1)}$.
 \end{fact}
\begin{proof}
We have $|\sigma( \mu ) | = | p^2/\lambda | < p^3/6$ for all $\sigma \in G$ and since $| \alpha_n | := p^{1-n} | a_n |,$ we get from 
\rf{abound} that 
\[ M_n < p^2 \cdot (p^2/6 )^n \cdot \bigg \vert \binom{l \frac{p-1}{p}}{n}\bigg \vert < \binom{n+l-1}{n}, \]
hence the claim. We note that the binomial coefficient behaves differently for various ranges of value for $l$; the values of interest are $l < p$,
so we see that $\binom{n+l-1}{n} < 4^n$ for all $l$ in the given range and $n < p$. For larger values of $n$, we use
the Stirling formula and apply it to the binomial coefficient value, which leads to the second bound.  
\end{proof}

We now proceed to the investigation of binomial power series considered as infinite vectors, and the possible linear relations among them.
The appropriate context for treating this question are the function field of $\K$ and extensions thereof.
We start by introducing some maps between $G$ orbits of elements in $\K$ and their rational coefficient vectors, and present the 
linear algebra of this context. Let 
\[ \rg{V} = \{ ( \sigma_c ( x ) )_{c=1}^{p-1} \ : \ x \in \K \} \subset \K^{(p-1)} \]
be the $\Q$-vector space of vectors of conjugates of numbers in $\K$. 
We let $\nu : \K \ra \rg{V}$ be the map $w \mapsto ( \sigma_c( w ))_{c=1}^{p-1} \in \rg{V}$ and 
$\kappa : \rg{V} \ra \Q^{(p-1)}$ be the coordinate map. For $v = \nu( w )$ and $w = \sum_c w_c \zeta^c$, the action is
\begin{eqnarray}
\label{nuk}
 \kappa( \nu( w ) ) = ( w_c )_{c = 1 }^{p-1} \in \Q^{p-1}, \quad \hbox{explicitly} \quad  w_c = \frac{1}{p} (\Tr( \zeta^{-c} w ) - \Tr( w ) )
\end{eqnarray}
The standard base of  $\Q^{p-1}$ is $\id{E} = \{ e_i \: \ i = 1, 2, \ldots, p-1\}$ with $e_i = ( \delta_{i,j} )_{j=1}^{p-1}$ and we let
$\Phi = \kappa^{-1}( \id{E} ) = \{ \Phi_i = \kappa^{-1}( e_i ) \ : \ i = 1,2, \ldots, p-1 \} $ be the induced standard base in $\rg{V}$:
it is the base built by the vectors $\nu( \zeta^i )$. 

Let $\id{T} = \{ \theta \in (1-\jmath) I \}$, which is a free $\Z$-module of rank $\frac{p-1}{2}$ 
generated by $(1-\jmath) \psi_n; n = 1,2, \ldots,\frac{p-1}{2}$. 
Let $i_p$ be the irregularity index of $p$, thus the number of odd integers $i < p-1$ such that the Bernoulli number 
$B_{p-i}$ is divisible by $p$; equivalently,
\[  e_{i} A[ p ] \neq \{ 1 \}, \quad \hbox{and} \quad \vartheta \cdot e_i \equiv 0 \bmod p. \]
These are precisely the components of the spectral decomposition of $\F_p[ G ]$ that annihilate $ I/ p I$ -- see also \cite{Wa}, \S 6.1.

We let $D = \frac{p-1}{2} - i_p$ and $(n_k)_{k=1}^D$ be a list of the odd indexes for which $B_{1, \omega^{-n_k} } \not \equiv 0 \bmod p$;
We write $\wh{\id{T}} =  \id{T}/( p \Z[ G ] \cap \id{T} )$; this $\F_p$-module is generated by the images of the Stickelberger elements
$\Theta'_n = (1-\jmath)(n-\sigma_n) \vartheta $
for $n = 2, 3, \ldots, \frac{p+1}{2}$. Since 
\[ \vartheta (n - \sigma_n ) e_k = (n-\sigma_n) B_{1, \varpi^{-k}} e_k \equiv 0 \bmod p \quad \Leftrightarrow \quad B_{p-k} \equiv 0 \bmod p, \]
if follows that $\prk( \wh{\id{T}} ) = \frac{p-1}{2} - i_p= D$. We used here classical formulae which can be found, 
for instance, in \cite{Wa}, p. 100-101.

Consider the function field $\K' = \K( T )$ and its extension $\KL' = \prod_{\theta \in \id{T}} \K'\left[ (1-\mu T )^{\theta/p} \right]$
with galois group $H = \Gal( \KL'/\K' )$. We define
\[ C = \cog( \KL'/\K' ) = \{ x \in {\KL' }^{\times} \ : \ x^p \in {\K'}^{\times}\} / {\K'}^{\times}, \]
the so called {\em cogalois} \cite{Al}  radical of the Kummer extension $\KL'/\K'$. If $B \subset {\K'}^{\times}$
 is the classical Kummer radical, then $C \cong B/(({\K'}^{\times})^p \cap B )$;
moreover, $C \cong H$ as finite abelian $p$-groups. Define now $\td{\Theta}_k = e_{n_k} \Theta_2; k = 1, 2, \ldots, D$. 
By definition, $\wh{\id{T}} = [ \td{\Theta}_k ; k=1, \ldots, D ]_{\F_p}$ and thus
\[ \KL' = \prod_{k=1}^D \K'\left[ ( 1- \mu T )^{\td{\Theta}_k/p} \right]; \quad C \cong \wh{\id{T}}. \]

In view of \rf{defbin}, there is an injective map $\iota : \KL' \ra \K'(( T ))$ induced by $(1-\mu T )^{\theta/p} \mapsto f[ \theta ]( T )$;
this extends to an injective map $ \iota' : \KL' \ra \K^{\N} \quad \hbox{ with } \quad  (1-\mu T )^{\theta/p} \mapsto \rg{a}( \theta )$. 
We note that the set 
$\{ 1 \} \cup \{ (1 - \mu T)^{c \td{\Theta}_k /p } \ : \ c = 1, 2, \ldots, p-1; k = 1, 2, \ldots, D \} $
builds a base of the $\K'$-vector space $\KL'/\K'$. Under the map $\iota'$, we deduce that the vectors in
\[ \id{A} = \left\{ \rg{a}( c \td{\Theta}_k ) \ : \ c = 1, 2, \ldots, p-1; k = 1, 2, \ldots, D  \right\} \]
are $\K$-independent. The set $\id{A}$ is closed under the action of $G$, and this action splits $\id{A}$ in mutually disjoint $G$-orbits;
there are thus $o_D = \frac{p^D-1}{p-1}$ such disjoint orbits and to each orbit $G \rg{a}( \theta )$ there belongs a set of $p-1$ vectors 
$\rg{u}^{(c)}( \theta )_{c=1}^{p-1} \in \left( \Z^{\N}\right)^{p-1} $; the connecting map here is the coordinate map $\kappa$ 
introduced on $\rg{V}$: it produces $\rg{u}( \theta ) = \kappa( \rg{a}(\theta) )$ by acting on the individual coefficient
vectors $\kappa : \nu( a_n[ \theta ] ) \mapsto \vec{u}_n( \theta )$ of the infinite matrices $\id{A}( \theta )$. 
It follows that the $\rg{u}( \theta )$ are consequently linearly independent too.

A fortiori, if $\id{F} \subset \wh{\id{T}}$ is any subset closed under the action of $G$, then the corresponding vector sets
\begin{eqnarray}
\label{vects}
 \id{A}( \id{F} ) = \{ \rg{a}( \theta ) \ : \ \theta \in \id{F} \} \quad \hbox{and} \quad \id{U}( \id{F} ) = 
\{ \rg{u}^{(c)}( \theta )_{c=1}^{p-1} \ : \ G \theta \subset \id{F} \} 
\end{eqnarray}
are linearly independent over $\K$ and $\Q$, respectively. We have proved:
\begin{proposition}
\label{rks}
We have the following equality of $p$-ranks:
\[ D := \frac{p-1}{2} - i_p = \prk( \wh{\id{T}} )  = \prk( \Gal( \KL'/\K' ) ) = H. \]
For any subset $\id{F} \subset \wh{\id{T}}$, the sets of infinite vectors  $  \id{A}( \id{F} ), \id{U}( \id{F} )$ 
defined in \rf{vects} are linearly independent over their respective fields of definition. 
\end{proposition}

\section{Lattices and linear algebra}
We consider the set 
\[ J_0 = \left\{ \sum_{j=1}^{(p-1)/2} c_j \psi_j \ : \ c_j \geq 0:  \ \sum_j c_j = p-1 \right\} ; \quad J = G\cdot J_0 \subset \wh{\id{T}} , \]
 in which $J_0$ is a set of linear combinations of the independent set of Fueter elements $\psi_j; j = 1,2,\ldots,(p-1)/2$
 and $J$ is its closure under the action of $G$: the closure will then contain $J_0$ and is made up of a number of mutually disjoint $G$-orbits.
 The number of elements can be estimated, with $q:= \frac{p-1}{2}$, assuming$p \geq 257$, and using the formula of Stirling, by:
 \begin{eqnarray}
 \label{bdn}
 N & = & | J | \geq | J_0 | = \binom{3 q - 1}{q-1} = \frac{1}{3} \binom{3q}{q} > \left( \frac{27}{4}\right)^q / 9 \sqrt{q} > (p-1) \cdot (5/2)^{p-1},\\
 N' &:= & N/(p-1) >  (5/2)^{p-1}. \nonumber
\end{eqnarray}

We focus on the {\em horizontal} vectors built from the coefficients of equal index in the 
vectors $\id{U}( J ) $ and build some large vector space by direct sums of copies of $\rg{V}$ associated to the orbits $G \theta \in J/G$.
We thus let 
\[ \bar{\rg{V}} = \rg{V}^{N'} = \bigoplus_{G \theta \in J/G} \rg{V}( \theta ) , \quad \rg{V}( \theta ) \cong \rg{V} , \forall \theta \]
 be the vector space built of concatenation of $N'$ vectors in $\rg{V}$, which can be identified
with $G$-orbits of elements in $\K$, and let $\rg{W} = \Q^{p-1}; \bar{\rg{W}} = \rg{W}^{N'}$. The maps $\nu, \kappa$ extend naturally
to maps
\[ \nu : \K^{N'} \ra \bar{\rg{V}}; \quad \kappa : \bar{\rg{V} } \ra \bar{\rg{W}}. \]
The standard base $\bar{\id{E}}$ of $\bar{\rg{W}}$ is the concatenation of $N'$ copies of $\id{E}$ and $\bar{\Phi} \subset \bar{\rg{V}}$, 
the induced base by the extended map $\kappa^{-1}$. It will be of help to associate the single isomorphic copies of $\rg{V}$ and $\rg{W}$
to the $G$-orbit $G \theta \subset \id{F}_k$ of some Stickelberger element, so 
\begin{eqnarray}
\label{concat}
 \bar{\rg{V}} = \bigoplus_{G \theta \subset \id{F}_k } \rg{V}( \theta ); \quad \bar{\rg{W}} = 
\bigoplus_{G \theta \subset \id{F}_k } \rg{W}( \theta )
\end{eqnarray}
We denote by accordingly
$\vec{v}_n = (\nu(a_n( \theta) ))_{ G \theta \subset \id{F}_k} \in \bar{\rg{V}}$ the row vectors built 
by the \nth{n} entries in  the vectors of $\id{U}( \id{F}_k )$ for the $G$-orbits of elements $\theta \in \wh{\id{T}}$.
Let $\rg{V}_m \subset \bar{\rg{V}}$ be the subspace spanned by the first $m$ row vectors $\vec{v}_n; n \leq m$. Since the infinite
vectors $\rg{a}( \theta ); \theta \in \id{F}_k$ are linearly independent, the vector space dimensions $d(m) = \dim_{\Q}(\rg{V}_m)$ are 
an increasing sequence --not necessarily strictly increasing -- that stabilizes at dimension $d(\infty) = N$: this is the column rank of the 
infinite matrix with rows $\vec{v}_n; n \in \N$, and the line rank is equal to it, be an elementary fact of linear algebra. For $m < N$ we 
let the discontinuities of the function $d : \N \ra [ 1 \ldots N ]$ be listed in the set
\[ S = \{ i_j \ : \ j = 0, 1, \ldots, t \leq N: d(i_j) < d(i_{j+1}); \hbox{ and $ d(i_j) = \ldots = d( i_{j+1} -1 )$} . \]  

We relate now the general theory developed so far to solutions of FLT2. We let $\mu = \frac{p^2}{\lambda}$ and $T = \frac{x+y}{y\cdot p^2}$;
then $f[ \theta ]( T )$ converges in $\Q_p[ \zeta ]$ to $\gamma[ \theta ]$. Moreover,
\begin{eqnarray*}
 \beta^2( \theta ) = s^{\varsigma( \theta ) } \cdot \gamma( \theta ) 
  = s^{\varsigma( \theta ) } \cdot \left( 1 + \sum_{n=1}^{\infty} a_n( \theta ) T^n\right); \quad a_n \in \Z[\zeta].
\end{eqnarray*}

As explained in the plan of the paper, we shall consider linear combinations of the $\beta^2[ \theta ]$ by some $\ell( \theta )\in \Z[\zeta ]$
yet to determine. By imposing galois covariance for the $\ell( \theta )$ , we will have 
\[ \sum_{c \in P^*} \ell( \sigma_c( \theta)) \cdot  \beta^2[ \sigma_c \theta ] = s^{2 \varsigma(\theta) }\Tr( \ell( \theta ) \gamma[\theta ] ).\]
We relate this linear combination to power series developments of the rational coefficients the $\beta$'s:
The series $\sum_{k=0}^{\infty} u_k( \theta )^{(c)}$ converge $p$-adically to rational numbers 
\[ u( \theta )^{(c)} \:= \sum_{k=0}^{\infty} u_k( \theta )^{(c)} \in \Q, \]
which are the  coefficients of 
\begin{eqnarray}
\label{cofs}
 \beta^2( \theta ) = s^{\varsigma( \theta )} \cdot \sum_{c=1}^{p-1} u( \theta)^{(c)} \zeta^c, \quad \hbox{for all $\theta \in J$.}
\end{eqnarray}
By the correspondence \rf{vecmap}, a linear combination 
\[ \sum_{c=1; G \theta \subset J}^{p-1} \nu'( c, \theta )  \rg{u}( \theta )^{(c)}; \quad \nu' \in \Z,\]
in which summation goes over the $p-1$ coefficient vectors of $f[ (1-\jmath) \theta ]( T )$ for representants $\theta \in J$ of all 
orbits $G \theta \subset J$ induces an explicit algebraic number
\begin{eqnarray}
\label{deldef}
 \delta := \sum_{ \theta \subset J / G } s^{2 \varsigma( \theta )} \Tr( \ell( \theta ) \overline{\gamma( \theta )} )\in \Z[ \zeta ]; \quad \ell \in \Z[\zeta ], 
 \end{eqnarray}
where $\ell( \theta )$ depend on the $\nu'( c, \theta )$ via \rf{vecmap}. 
We intend to choose the $\ell( \theta )$ such that $\delta \equiv 0 \bmod w^m$, for a large value of $m$ and $w = p^{v_p( T )}$, 
together with a proof that $\delta \neq 0$. We also wish to keep the coefficients
$\ell( \theta )$ relatively small; for instance, in the order of magnitude of $| a_m( \theta ) |$.

This will be done as follows: let 
\begin{eqnarray}
\label{param}
 m' = \lfloor N'/2 \rfloor; n' = \lfloor N'/p \rfloor; \quad m = (p-1) m'; n= (p-1) n. 
 \end{eqnarray}
 Let $R_m = \min \{ s \in J \:\ s \geq m-n \}$ und $R_n = \min \{ s \in S \ : \ s \geq m \}$ be the largest integers for which the 
dimensions $d( R_m - 1 ) = d( m -n )$ and $d( R_n - 1 ) = d( n )$, and the dimension has a jump at those indices. 

We let $v = v_{R_m}$ and choose $\Phi' \in \bar{\Phi}$ such that $v + \Phi' \not \in
V_{R_m}$. More precisely, $\Phi' = ( \Phi' ( \theta ) )_{ G \theta \subset \id{F}_k }$ and there is a $\psi \in \id{F}_k$
such that all components $\Phi'( \theta ) = 0$ for $\theta \neq \psi$ while $\Phi'( \psi ) = \nu( \zeta^{j} )$ for some $j \in P^*$. 

Since $\bar{\Phi}$ is a base for $\bar{\rg{V}}$ and 
$d( R_m ) < \dim(\bar{\rg{V}})$, such a base vector necessarily must exist.
We then let $v'_j = v_j$ for all $j \leq R_m$ with the exception of $R_n$, and let
$v'_{R_n} = v_{R_n} + \Phi'$. Let $V'_l = [ v'_j ; j \leq l ]_{\Q}$ for all $l \leq R_n$,
be the span and $W'_l = \kappa( V'_l )$. We now select in $W'_{R_n}$ a set of 
$\bar{d}:=d( R_n )$ vectors among the $\{ w'_j ; \ j \leq R_n \}$, say 
$\omega_j ; j = 1, \ldots, \bar{d}$ such that $\kappa( \omega_j )$  span the space $W'_{R_n}$; 
we may assume that the indices $j$ are the smallest among all possible choices, and then they
will also be elements in $S$ and let $h \leq \bar{d}-1$ be such that $\omega_h = v'_{R_m}$.
For $i \in \{ R_n+1, \ldots , R_m \}$ we let the coefficients of $v_{R_n}$ in the development of 
$v_i$ in the base of the $\omega_j$ be $\chi_{i-R_n}$. Thus $v_i - \chi_{i-R_n}$ is in the span $[ \omega_j ; j \neq R_n ]_{\Q}$. 

Let now $A$ be the matrix having the $\kappa(\omega_j)$ as row vectors, thus 
$A \in \Mat( \bar{d} , N )$ and $\bar{d} < N/2 $. Our solution is based on finding a 
short non trivial solution $w \in \Z^N$ of the homogenous linear system $A w = 0$, using the 
the Siegel box principle. We let $L = (\ell(\theta))_{G \theta \in \id{F}_k } = \phi^{-1}( w )$. 
Note that the choice of $w'_{R_m}$ and the definition of $\Psi'$ guarantee that 
\begin{eqnarray}
\label{inhom}
H & = & \Tr ( L \cdot \ol{( v_{R_m} )} ) =  - \Tr( L \cdot \ol( \Phi' ) ) = - \Tr( \zeta{-j} \ell( \psi ) ) \neq 0, \nonumber \\
\quad \quad \delta & = & \Tr\left( \sum_{G \theta \in \id{F}_k } s^{2\varsigma( \theta )} \ell( \theta ) \cdot \ol{\beta}( \theta ) \right)  \\
& = & T^{R_m} H \cdot \left( 1 + T \sum_{i = 1}^{m-n}  T^{i-1} \chi_i \right) + O( T^{m+1} ) =: 
T^{R_m} H \cdot U + O( T^{m+1} ); \  U \in \Z_p^{\times}; \nonumber
\end{eqnarray}
here we designated the sum in the brackets by $U$;since $v_p( T ) > 0$, this is a $p$-adic unit.
In order to complete the proof, we need to estimate $H$ and $|| L ||$, and show that 
the choices of $m, n$ imply that $H < T^n/2$ and thus $\delta \neq 0$. This then leads to the lower bounds.

\section{Lower bounds for FLT2} 
We keep the notations introduced at the end of the previous chapter, let $k = p-1$, so by \rf{bdn},
 $N/(p-1) = N' >  (5/2)^{ p-1 }$ for $p \geq 257$, say. 

 The bounds in Fact \ref{bounds} induce the generous upper bound 
\[ || A ||_1 \leq (N/2)^{q} \cdot (2p/3)^{N+2} < \frac{p^N}{(2p+1)^2 N}  =: M . \]  
 for all the entries of $A$. By the Siegel box principle, there is a small solution of $A w = 0$, 
that verifies
\[ | w |_1 = \max_{c, \theta} | w^{(c)}( \theta ) | < N \cdot M < \frac{p^N}{(2p+1)^2} =: L ,    \] 
and thus $|| L ||_1 = \max_{G \theta \subset \id{F}_k} || \ell( \theta ) ||_1 < L$ and
thus $H < L \cdot (p-1)  < p^{ (m-n) v_p ( T ) }$. Assume that $\delta = 0$; then $\delta/( T^{R_m} U ) = 0$ and
we gather from \rf{inhom} that $H = O( T^{m+1-R_m} )$, in contradiction with our bound on $H$. Therefore $\delta \neq 0$.
Since $\delta \equiv 0 \bmod T^{R_n}$ it follows a fortiori that 
\[  | \delta | \geq p^{(2p-3) \cdot ( N/2 - N/p)} > p^{( p-4 ) \cdot N}. \]
From the definition \rf{deldef} and \rf{inhom}, we find 
\[ |  \delta | \leq L \cdot N \cdot s^{2(p-1)}, \] 
and by comparing the two bounds, we finally find
\[ | s | > \left( \frac{p^{N(p-4)} }{L N}\right)^{1/2(p-1) } > p^{N \frac{p-5}{2(p-1)}} = p^{N/2-N/q} > p^{(5/2)^{p-1}}. \] 
Herewith, Theorem \ref{main} follows:
 \begin{proof}
 By definition, $| s^p | = \frac{| x^p + y^p |}{| x+ y | } \leq \max( | x |^p, | y |^p ) \leq \max( | x |^p, |y |^p, | z |^p ) $ and from the above
 bound for $|s |$ we conclude 
 \[ p^{(5/2)^{p-1}} < | s | \leq \max( | x |, | y |, | z | ) ,\]
 hence the claim.
  \end{proof}


\begin{thebibliography}{ZZZZZZ}
 \bibitem[Al]{Al} T.~Albu.
\newblock {\em Cogalois theory}.
\newblock Number 252 in Monographs and textbooks in pure and applied
  mathematics. Marcel Dekker Inc., 2003.



%





\bibitem[Fe]{Fe} Ivan Fesenko: {\em Personal communication, December 2020}

\bibitem[FGK]{FGK} L. Fukshansky, P.Guerzhoy and S. K\"uhnlein:{\em On sparse geometry of numbers}, 
available at https://www1.cmc.edu/pages/faculty/lenny/papers/sparse\_geometry.pdf


\bibitem[HHO]{HHO} W. Hart,  D. Harvey and W. Ong: {\em Irregular primes to two billion}, Arxive, 1605.02398v1.

  

\bibitem[Iw]{Iw} K. Iwasawa: {\em A Note on Jacobi Sums},
Symp. Math., {\bf 15}, (1975), pp. 447 - 459.

\bibitem[Jh]{Jh} Vijay Jha: {\em The Stickelberger Ideal
in the Spirit of Kummer with Applications to the First
Case of Fermat's Last Theorem}, Queen's papers
in pure and applied mathematics, {\bf 93}, Kingston Ontario, (1993).

\bibitem[La1]{La1} Lang. S.: {\em Cyclotomic Fields, I and II},
  Combined second edition with an Appendix by Karl Rubin, 
  Graduate Texts in Mathematics \textbf{121}, Springer (1990)
 
  \bibitem[Le]{Le} T. Lepist\"o: {\em On the growth of the first 
factor of the class number of the prime cyclotomic field}, 
Ann. Acad. Sci. Fenn., Ser A1 Math. {\bf 577} (1974).

\bibitem[Mi]{Mi} P. Mih\u{a}ilescu: {\em Primary Cyclotomic Units
and a Proof of Catalan's Conjecture}, J. Reine Angew. Math. 572 (2004), 167--195

\bibitem[Mi2]{Mi2} P. Mih\u{a}ilescu: {\em Class Number Conditions for the Diagonal Case of the Equation
of Nagell and Ljunggren}, In {\em Festschrift to the 70-th Birthday of Wolfgang Schmidt},Eds. Schlickewei et. al, Springer (2008), pp. 243-274.

\bibitem[Mi3]{Mi3} P. Mih\u{a}ilescu: {\em Improved lower bounds for possible solutions in the Second Case of the Fermat Last Theorem 
and in the Catalan Equation}, to appear in Journal of Number Theory.

\bibitem[Mll]{Mll} J. Milne: {\em Algebraic Number Theory} https://www.milne.org/math/CourseNotes/ant.html

\bibitem[Mo]{Mo} S. Mochizuki, {\em Inter-universal Teichm\"uller Theory I,II,III,IV}, accepted for publication and to appear in Publ. Res. Inst. Math. Sci. 57 (2021), 
{\em for an  announcements of the EMS see https://ems.press/updates/2020-11-16-prims-special-issues-2021 }.

\bibitem[MFHMP]{MFHMP} S. Mochizuki, I. Fesenko, Y. Hochi, A. Minamide, W. Porowski,
{\em Explicit estimates in inter-universal Teichmüller theory} ,preprint 2020

\bibitem[Ri1]{Ri1} P. Ribenboim: {\em $13$ Lectures on Fermat's Last Theorem}, Springer Verlag (1979).

\bibitem[Wa]{Wa} L. Washington: {\em Introduction to 
Cyclotomic Fields}, Second Edition, Springer (1996), 
Graduate Texts in Mathematics {\bf 83}.

\bibitem[W]{W} Wiles, Andrew: {\em Modular elliptic curves and Fermat's Last Theorem}. Annals of Mathematics. (1995), 141 (3): 443–551. 

\bibitem[WT]{WT} Taylor, R. and  Wiles, A. :  {\em Ring theoretic properties of certain Hecke algebras.} Annals of Mathematics. (1995). 141 (3): 553–572. 

\end{thebibliography}
\end{document}